\documentclass{amsart}

\usepackage{amsmath,amssymb,amsthm}
\usepackage{a4wide}
\usepackage{hyperref}

\usepackage{graphicx}
\usepackage{xypic}
\entrymodifiers={+!!<0pt,\fontdimen22\textfont2>}
\usepackage[all]{xy}

\usepackage{multirow,bigdelim}

\newtheoremstyle{myremark} 
    {7pt}                    
    {7pt}                    
    {}  	                 
    {}                           
    {\bf}       	         
    {.}                          
    {.5em}                       
    {}  

\theoremstyle{plain}
\newtheorem{lemma}{Lemma}[section]
\newtheorem{theorem}[lemma]{Theorem}

\newtheorem{definition}[lemma]{Definition}

\newtheorem{proposition}[lemma]{Proposition}
\newtheorem{conjecture}[lemma]{Conjecture}

\theoremstyle{myremark}
\newtheorem{remark}[lemma]{Remark}
\newtheorem{example}[lemma]{Example}

\newcommand{\zet}{\mathbb{Z}}

\newcommand{\ind}{\mathrm{Ind}}
\newcommand{\htpyequiv}{\simeq}

\newcommand{\neck}{\mathrm{Neck}}
\newcommand{\prop}{\mathrm{Prop}}
\newcommand{\incl}{\hookrightarrow}
\renewcommand{\subset}{\subseteq}
\newcommand{\susp}{\Sigma}
\newcommand{\Michal}[1]{}

\newcommand{\empha}[1]{\textbf{#1}}		
\newcommand{\nicep}{\mathcal{P}}
\newcommand{\nicei}{\mathcal{I}}

\begin{document}
\title{Hard squares on cylinders revisited}

\author[Micha{\l} Adamaszek]{Micha{\l} Adamaszek}
\address{Mathematics Institute and DIMAP,
      \newline University of Warwick, Coventry, CV4 7AL, UK}
\email{aszek@mimuw.edu.pl}
\thanks{Research supported by the Centre for Discrete
        Mathematics and its Applications (DIMAP), EPSRC award EP/D063191/1.}

\subjclass[2010]{}

\begin{abstract}
We consider the independence complexes of square grids with cylindrical boundary conditions. When one of the dimensions is small we use simple reductions induced by edge removals to show explicit natural homotopy equivalences between those spaces. In the second part we expand the results of Jonsson, who calculated the Euler characteristic of cylinders with odd circumference. We describe a series of results for cylinders of even circumference. Finally we define a completely independent combinatorial model (necklaces) which calculates the generating functions of the Euler characteristic of cylindrical grids. We conjecture that this model has some particularly simple structure.
\end{abstract}
\maketitle

\section{Introduction}
\label{sect:intro}

This paper is motivated by a recent collection of papers \cite{JonCylinder,JonRhombus,BLN,EngWitten,Thapper,JonDiag}, which in turn were motivated by some combinatorial questions in statistical physics, namely the properties of the Witten index in the so-called hard squares model \cite{Fen}. 

In the hard-core model on a grid we are given a graph $G$ and particles (fermions) located in the vertices of the grid, with the restrictions that two adjacent vertices cannot be occupied at the same time. That leads to considering the \emph{independent sets} in $G$, which represent the allowed states of the system. The family of all independent sets in $G$ naturally forms a simplicial complex, the \emph{independence complex} $\ind(G)$ of $G$. Various topological invariants of $\ind(G)$ correspond to physical characteristics of the underlying hard-core model. Among those are e.g. homology \cite{EngWitten,AdaSuper,Eer,Fen,HuSch2,JonGrids,HuOthers,HuSch1} or the Euler characteristic which is re-branded as the Witten index \cite{JonCylinder,JonRhombus,BLN,EngWitten,Thapper,JonDiag,Fen2}. Sometimes it is even possible to determine the exact homotopy type of $\ind(G)$ \cite{BLN,Thapper,HHFS,HuOthers}.

In this work we continue the research line of Jonsson studying these spaces for the square grids with various boundary conditions. Let $P_m$ denote the path with $m$ vertices and let $C_n$ be the cycle with $n$ vertices. The free square grid is $G=P_m\times P_n$, its cylindrical version is $G=P_m\times C_n$ and the toroidal one is $G=C_m\times C_n$. 

In the first part we show natural recursive dependencies in all three models. We concentrate mainly on cylinders.

\begin{theorem}
\label{thm:periods}
We have the following homotopy equivalences in the cylindrical case:
\begin{itemize}
\item[a)] $\ind(P_1\times C_n)\htpyequiv \susp\,\ind(P_1\times C_{n-3})$,
\item[b)] $\ind(P_2\times C_n)\htpyequiv \susp^2\,\ind(P_2\times C_{n-4})$,
\item[c)] $\ind(P_3\times C_n)\htpyequiv \susp^6\,\ind(P_3\times C_{n-8})$,
\item[d)] $\ind(P_m\times C_3)\htpyequiv \susp^2\,\ind(P_{m-3}\times C_3)$,
\item[e)] $\ind(P_m\times C_5)\htpyequiv \susp^2\,\ind(P_{m-2}\times C_5)$,
\item[f)] $\ind(P_m\times C_7)\htpyequiv \susp^6\,\ind(P_{m-4}\times C_7)$,
\end{itemize}
where $\susp$ denotes the unreduced suspension.
\end{theorem}

Here a) is a classical result of Kozlov \cite{Koz}, while d) and e) also follow from \cite{Thapper} where those spaces were identified with spheres by means of explicit Morse matchings. The results b), c) and f) are new and were independently proved by a different method in \cite{Kouyemon}. Note that a), b), c) are `dual' to, respectively, d), e) and f) in the light of Thapper's conjecture \cite[Conj. 3.1]{Thapper} (also \cite[Conj. 1.9]{Kouyemon}) that $\ind(P_m\times C_{2n+1})\htpyequiv \ind (P_n\times C_{2m+1})$. If one assumes the conjecture holds, then a), b) and c) are equivalent to, respectively, d), e) and f). Theorem \ref{thm:periods} together with an easy verification of initial conditions imply the conjecture for $m\leq 3$. The statements d), e) and f) take the periodicity of Euler characteristic, proved by Jonsson \cite{JonCylinder}, to the level of homotopy type.

According to the calculation of K. Iriye \cite{KouyemonPriv} there is an equivalence $\ind(P_4\times C_{2k+1})\htpyequiv\ind(P_k\times C_9)$ and both spaces are, up to homotopy, wedges of spheres. However, the number of wedge summands grows to infinity as $k\to\infty$, so no recursive relation as simple as those in Theorem~\ref{thm:periods} is possible for $\ind(P_m\times C_9)$ nor $\ind(P_4\times C_n)$.

Let us also mention that a completely analogous method proves the following.
\begin{proposition}
\label{prop:other}
We have the following homotopy equivalences in the free and toroidal cases:
\begin{itemize}
\item $\ind(P_1\times P_n)\htpyequiv \susp\,\ind(P_1\times P_{n-3})$,
\item $\ind(P_2\times P_n)\htpyequiv \susp\,\ind(P_2\times P_{n-2})$,
\item $\ind(P_3\times P_n)\htpyequiv \susp^3\,\ind(P_3\times P_{n-4})$,
\item $\ind(C_3\times C_n)\htpyequiv \susp^2\,\ind(C_3\times C_{n-3})$.
\end{itemize}
\end{proposition}

All those results are proved in Section \ref{sect:simplesuspension}.

In the second part of this work we aim to provide a method of recursively calculating the Euler characteristic in the cylindrical case $P_m\times C_n$ when the circumference $n$ is even. Since it is customary to use the Witten index in this context, we will adopt the same approach and define for any space $X$
$$Z(X)=1-\chi(X)$$
where $\chi(X)$ is the unreduced Euler characteristic. Then $Z(X)=0$ for a contractible $X$, $Z(\susp\,X)=-Z(X)$ for any finite simplicial complex $X$, and $Z(S^k)=(-1)^{k-1}$. Then the value
$$Z(G):=Z(\ind(G))$$
is what is usually called the Witten index of the underlying grid model.

Table 1 in the appendix contains some initial values of $Z(P_m\times C_n)$ arranged so that $m$ labels the rows and $n$ labels the columns of the table. Let
$$f_n(t)=\sum_{m=0}^\infty Z(P_m\times C_{n})t^m$$
be the generating function of the sequence in the $n$-th column. By an ingenious matching Jonsson \cite{JonCylinder} computed the numbers $Z(P_m\times C_{2n+1})$ for odd circumferences and found that for each fixed $n$ they are either constantly $1$ or periodically repeating $1,1,-2,1,1,-2,\ldots$. Precisely
\begin{eqnarray*}
f_{6n+1}(t)=f_{6n-1}(t)&= &\frac{1}{1-t},\\
f_{6n+3}(t)&=&\frac{1-2t+t^2}{1-t^3}.
\end{eqnarray*}
The behaviour of $Z(P_m\times C_{2n})$ is an open problem of that work, which we tackle here. Also, recently Braun notes that some problems faced in \cite{Braun} are reminiscent of the difficulty of determining the homotopy types of the spaces $\ind(P_m\times C_{2n})$.

Our understanding of the functions $f_{2n}(t)$ comes in three stages of increasing difficulty.
\begin{theorem}
\label{thm:genfun}
Each $f_{2n}(t)$ is a rational function, such that all zeroes of its denominator are complex roots of unity.
\end{theorem}
Our method also provides an algorithm to calculate $f_{2n}(t)$, see Appendix. This already implies that for each fixed $n$ the sequence $a_m=Z(P_m\times C_{2n})$ has polynomial growth. However, we can probably be more explicit:

\begin{conjecture}
\label{thm:genfunmedium}
For every $n\geq 0$ we have
\begin{eqnarray*}
f_{4n+2}(t)&= &\frac{h_{4n+2}(t)}{(1+t^2)\cdot\big[(1-t^{8n-2})(1-t^{8n-8})(1-t^{8n-14})\cdots(1-t^{2n+4})\big]},\\
f_{4n}(t)&=&\frac{h_{4n}(t)}{(1-t^2)\cdot\big[(1-t^{8n-6})(1-t^{8n-12})(1-t^{8n-18})\cdots(1-t^{2n+6})\big]}.
\end{eqnarray*}
for some polynomials $h_{n}$.
\end{conjecture}
In the denominators the exponents decrease by $6$.

Conjecture \ref{thm:genfunmedium} is in fact still just the tip of an iceberg, because $p_n(t)$ turn out to have lots of common factors with the denominators. This leads to the grande finale:
\begin{conjecture}
\label{conjecture:periodic}
After the reduction of common factors:
\begin{itemize}
\item $f_{4n+2}(t)$ can be written as a quotient whose denominator has no multiple zeroes. Consequently, for any fixed $n$, the sequence $a_m=Z(P_m\times C_{4n+2})$ is periodic.
\item $f_{4n}(t)$ can be written as a quotient whose denominator has only double zeroes. Consequently, for any fixed $n$, the sequence $a_m=Z(P_m\times C_{4n})$ has linear growth.
\end{itemize}
\end{conjecture}

A direct computation shows that the first part holds for a number of initial cases, with periods given by the table:
\begin{center}
\begin{tabular}{l||l|l|l|l|l|l}
$4n+2$ & 2 & 6 & 10 & 14 & 18 & 22\\ \hline
period & 4 & 12 & 56 & 880 & 360 & 276640
\end{tabular}
\end{center}

In Section \ref{sect:patterns} we prepare our main tool for the proof of Theorem~\ref{thm:genfun}: \emph{patterns} and their $\mu$-invariants. A small example of how they work is presented in detail in Section~\ref{sect:example6}. The proof of Theorem~\ref{thm:genfun} then appears in Section~\ref{sect:weakproof}. In Section~\ref{section:neck} we describe a completely independent combinatorial object, the \emph{necklace graph}, which is a simplified model of interactions between patterns. It has some conjectural properties, esp. Conjecture~\ref{con:cyc-len}, whose verification would prove Conjecture~\ref{thm:genfunmedium}. Section~\ref{section:neck}, up to and including Conjecture~\ref{con:cyc-len}, can be read without any knowledge of any other part of this paper. As for Conjecture~\ref{conjecture:periodic}, it seems unlikely that the methods of this paper will be sufficient to prove it.

To avoid confusion we remark that our results are in a sense orthogonal to some questions raised by Jonsson, who asked if the sequence in each \emph{row} of Table 1 (see Appendix) is periodic. That question is equivalent to asking if the eigenvalues of certain transfer matrices are complex roots of unity, and this paper is not about them.

\section{Prerequisites on homotopy of independence complexes}
\label{sect:prereq}

We consider finite, undirected graphs. If $v$ is a vertex of $G$ then by $N(v)$ we denote the \emph{neighbourhood} of $G$, i.e. the set of adjacent vertices, and by $N[v]$ the \emph{closed neighbourhood}, defined as $N[v]=N(v)\cup\{v\}$. If $e=\{u,v\}$ is any pair of vertices (not necessarily an edge), then we set $N[e]=N[u]\cup N[v]$. 

If $G\sqcup H$ is the disjoint union of two graphs then its independence complex satisfies
$$\ind(G\sqcup H)=\ind(G)\ast\ind(H)$$
where $\ast$ is the join. In particular, if $\ind(G)$ is contractible then so is $\ind(G\sqcup H)$ for any $H$. We refer to \cite{Book} for facts about (combinatorial) algebraic topology. By $\susp K=S^0\ast K$ we denote the suspension of a simplicial complex $K$.

The independence complex is functorial in two ways. First, for any vertex $v$ of $G$ there is an inclusion $\ind(G\setminus v)\incl \ind(G)$. Secondly, if $e$ is an edge of $G$ then the inclusion $G- e\incl G$ induces an inclusion $\ind(G)\incl\ind(G- e)$. This leads to two main methods of decomposing simplicial complexes, using vertex or edge removals. They can be analyzed using the two cofibration sequences:
\begin{eqnarray*}
\ind(G\setminus N[v])\incl\ind(G\setminus v)\incl\ind(G)\to\susp\,\ind(G\setminus N[v])\to\cdots,\\
\susp\,\ind(G\setminus N[e])\incl\ind(G)\incl\ind(G- e)\to\susp^2\,\ind(G\setminus N[e])\to\cdots,
\end{eqnarray*}
both of which are known in various forms. See \cite{Ada} for proofs.

The following are all immediate consequences of the cofibration sequences.
\begin{lemma}
\label{lemma:simple}
We have the following implications
\begin{itemize}
\item[a)] If $\ind(G\setminus N[v])$ is contractible then $\ind(G)\htpyequiv \ind(G\setminus v)$.
\item[b)] If $\ind(G\setminus v)$ is contractible then $\ind(G)\htpyequiv \susp\,\ind(G\setminus N[v])$.
\item[c)] If $\ind(G\setminus N[e])$ is contractible then $\ind(G)\htpyequiv \ind(G- e)$.
\end{itemize}
\end{lemma}

\begin{lemma}
\label{lemma:zadditive}
For any vertex $v$ and edge $e$ of $G$ we have
\begin{eqnarray*}
Z(G)&=&Z(G\setminus v)-Z(G\setminus N[v]),\\
Z(G)&=&Z(G- e)-Z(G\setminus N[e]).
\end{eqnarray*}
\end{lemma}

There are some special combinatorial circumstances where Lemma \ref{lemma:simple} applies. 
\begin{lemma}
\label{lemma:simple2}
We have the following implications
\begin{itemize}
\item[a)] (Fold lemma, \cite{Eng1}) If $N(u)\subset N(v)$ then $\ind(G)\htpyequiv \ind(G\setminus v)$.
\item[b)] If $u$ is a vertex of degree $1$ and $v$ is its only neighbour then $\ind(G)\htpyequiv \susp\,\ind(G\setminus N[v])$.
\item[c)] If $u$ and $v$ are two adjacent vertices of degree $2$ which belong to a $4$-cycle in $G$ together with two other vertices $x$ and $y$ then $\ind(G)\htpyequiv \susp\,\ind(G\setminus\{u,v,x,y\})$ (see Fig.\ref{fig:1}).
\item[d)] If $G$ is a graph that contains any of the configurations shown in Fig.\ref{fig:contr}, then $\ind(G)$ is contractible.
\end{itemize}
\end{lemma}
\begin{proof}
In a) vertex $v$ satisfies Lemma \ref{lemma:simple}.a) and in b) it satisfies Lemma \ref{lemma:simple}.b). In c) one can first remove $x$ without affecting the homotopy type (because $G\setminus N[x]$ has $u$ as an isolated vertex), and then apply Lemma \ref{lemma:simple} to $v$. Finally d) follows because a single operation of type described in b) or c) leaves a graph with an isolated vertex.
\end{proof}

A vertex $v$ of $G$ is called \emph{removable} if the inclusion $\ind(G\setminus v)\incl \ind(G)$ is a homotopy equivalence. We call the graph $G\setminus N[v]$ the \emph{residue graph} of $v$ in $G$. By Lemma \ref{lemma:simple}.b if the residue graph of $v$ has a contractible independence complex then $v$ is removable.

If $e$ is an edge of $G$ then we say $e$ is \emph{removable} if the inclusion $\ind(G)\incl \ind(G- e)$ is a homotopy equivalence. If $e$ is not an edge of $G$ then $e$ is \emph{insertable} if the inclusion $\ind(G\cup e)\incl\ind(G)$ is a homotopy equivalence or, equivalently, if $e$ is removable from $G\cup e$. In both cases we call the graph $G\setminus N[e]$ the \emph{residue graph} of $e$ in $G$. By Lemma \ref{lemma:simple}.c if the residue graph of $e$ has a contractible independence complex then $e$ is removable or insertable, accordingly.

\begin{figure}
\begin{center}
\includegraphics[scale=0.85]{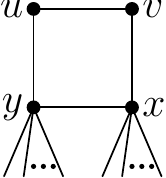}
\end{center}
\caption{A configuration which can be removed at the cost of a suspension (Lemma \ref{lemma:simple2}.c).}
\label{fig:1}
\end{figure}

\begin{figure}
\begin{center}
\begin{tabular}{ccccccc}
\includegraphics[scale=0.85]{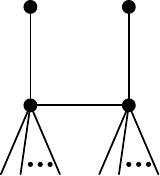} & &\includegraphics[scale=0.85]{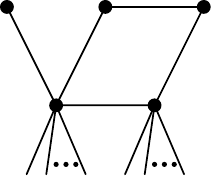} & & \includegraphics[scale=0.85]{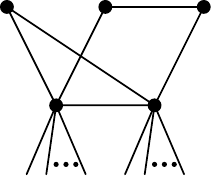} & & \includegraphics[scale=0.85]{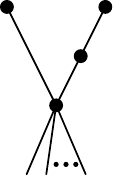} \\
A & & B & & C & & D
\end{tabular}
\end{center}
\caption{Four types of configurations which force contractibility of the independence complex (Lemma \ref{lemma:simple2}.d).}
\label{fig:contr}
\end{figure}

\section{Proofs of Theorem \ref{thm:periods} and Proposition \ref{prop:other}}
\label{sect:simplesuspension}

We identify the vertices of $P_m$ with $\{1,\ldots,m\}$ and the vertices of $C_n$ with $\zet/n=\{0,\ldots,n-1\}$. Product graphs have vertices indexed by pairs. To deal with the degenerate cases it is convenient to assume that $C_2=P_2$, that $C_1$ is a single vertex with a loop and that $C_0=P_0$ are empty graphs. This convention forces all spaces $\ind(C_1)$, $\ind(P_0)$ and $\ind(C_0)$ to be the empty space $\emptyset=S^{-1}$.

\begin{proof}[Proof of \ref{thm:periods}.a).]
This is a known result of \cite{Koz}, but we present a proof that introduces our method. Identify $P_1\times C_n$ with $C_n$. Let $e=\{0,4\}$. The residue graph of $e$ has $2$ as an isolated vertex, so $e$ is insertable. In $C_n\cup e$ the edge $\{0,1\}$ is removable (because $3$ is isolated in the residue) and subsequently $\{3,4\}$ is removable ($1$ isolated in the residue). It follows that
$$\ind(C_n)\htpyequiv\ind(C_n\cup\{0,4\}\setminus\{0,1\}\setminus\{3,4\}).$$
But the latter graph is $C_{n-3}\sqcup P_3$, therefore its independence complex is $\ind(C_{n-3})\ast\ind(P_3)\htpyequiv \ind(C_{n-3})\ast S^0=\susp\,\ind(C_{n-3})$ as required.
\end{proof}

\begin{proof}[Proof of \ref{thm:periods}.b).]
In $P_2\times C_n$ the edge $e_1=\{(1,0),(1,5)\}$ is insertable, because its residue graph contains a configuration of type A (see Lemma \ref{lemma:simple2}.d), namely with vertices $(2,1)$ and $(2,4)$ having degree one. For the same reason the edge $e_2=\{(2,0),(2,5)\}$ is insertable. Now in the graph $(P_2\times C_n)\cup\{e_1,e_2\}$ the edges 
$$f_1=\{(1,0),(1,1)\}, f_2=\{(2,0),(2,1)\}, f_3=\{(1,4),(1,5)\}, f_4=\{(2,4),(2,5)\}$$
are all sequentially removable, because the residue graph in each case contains a configuration of type B. Therefore
\begin{eqnarray*}
\ind(P_2\times C_n)&\htpyequiv & \ind((P_2\times C_n)\cup\{e_1,e_2\}-\{f_1,f_2,f_3,f_4\})=\\
&=&\ind(P_2\times C_{n-4}\sqcup P_2\times P_4)=\ind(P_2\times C_{n-4})\ast\ind(P_2\times P_4)\htpyequiv\\
&\htpyequiv &\ind(P_2\times C_{n-4})\ast S^1= \susp^2\,\ind(P_2\times C_{n-4})
\end{eqnarray*}
where $\ind(P_2\times P_4)$ can be found by direct calculation or from Proposition \ref{prop:other}.
\end{proof}

\begin{remark}
All the proofs in this section will follow the same pattern, that is to split the graph into two parts. One of those parts will be small, i.e. of some fixed size, and its independence complex will always have the homotopy type of a single sphere. Every time we need to use a result of this kind about a graph of small, fixed size, we will just quote the answer, leaving the verification to those readers who do not trust computer homology calculations \cite{Poly}.
\end{remark}

\begin{proof}[Proof of \ref{thm:periods}.c).]
We follow the strategy of b). First we need to show that the edges $e_1=\{(1,0),(1,9)\}$, $e_2=\{(2,0),(2,9)\}$, $e_3=\{(3,0),(3,9)\}$ are insertable. 

For $e_1$ the residue graph is shown in Fig.\ref{fig:p3cn}.a. To prove its independence complex is contractible we describe a sequence of operations that either preserve the homotopy type or throw in an extra suspension. The sequence will end with a graph whose independence complex is contractible for some obvious reason, so also the complex we started with is contractible. The operations are as follows: remove $(3,2)$ (by \ref{lemma:simple2}.a with $u=(2,1)$); remove $(2,3)$ (by \ref{lemma:simple2}.a with $u=(1,2)$); remove $N[(3,4)]$ (by \ref{lemma:simple2}.b with $u=(3,3)$); remove $(2,6)$ (by \ref{lemma:simple2}.a with $u=(1,7)$); remove $N[(1,5)]$ (by \ref{lemma:simple2}.b with $u=(2,5)$). In the last graph there is a configuration of type A on the vertices $(3,6)$ and $(1,7)$.

The same argument works for $e_3$. For $e_2$ the residue graph has a connected component shown in Fig.\ref{fig:p3cn}.b. The modifications this time are: remove $N[(1,2)]$, $N[(1,7)]$, $N[(3,2)]$  and $N[(3,7)]$ for reasons of \ref{lemma:simple2}.b. The graph that remains contains a configuration of type A.

Next it remains to check that in $(P_3\times C_n)\cup\{e_1,e_2,e_3\}$ the edges $\{(0,i),(1,i)\}$ and $\{(8,i),(9,i)\}$ are removable for $i=1,2,3$. The types of residue graphs one must consider are quite similar and the arguments for the contractibility of their independence complexes are exact copies of those for $e_1,e_2,e_3$ above. We leave them as an exercise to the reader.

We can thus conclude as before
\begin{equation*}
\ind(P_3\times C_n)\htpyequiv\ind(P_3\times C_{n-8})\ast\ind(P_3\times P_8)\htpyequiv\ind(P_3\times C_{n-8})\ast S^5= \susp^6\,\ind(P_3\times C_{n-8}).
\end{equation*}

\end{proof}

\begin{figure}
\begin{center}
\begin{tabular}{ccc}
\includegraphics[scale=0.85]{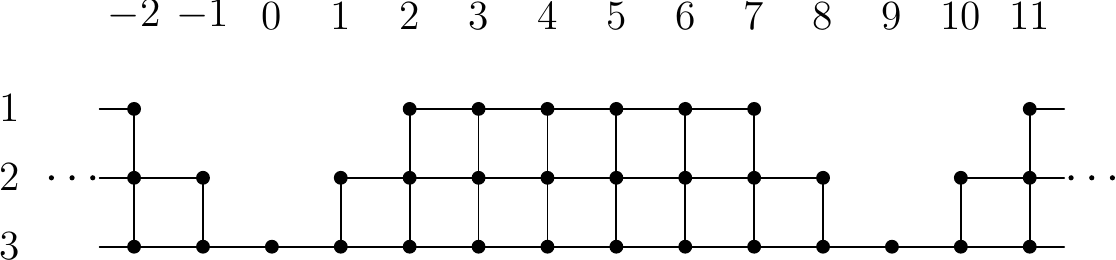} & &\includegraphics[scale=0.85]{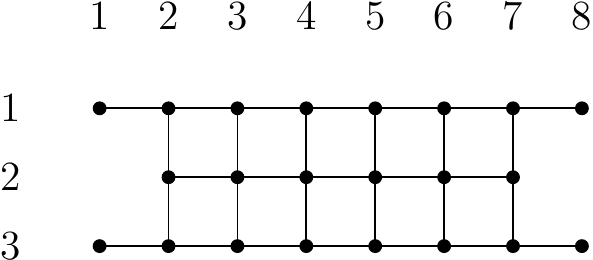} \\
a) & &  b)
\end{tabular}
\end{center}
\caption{Two residue graphs for edges insertable into $P_3\times C_n$.}
\label{fig:p3cn}
\end{figure}

\begin{figure}
\begin{center}
\includegraphics[scale=0.85]{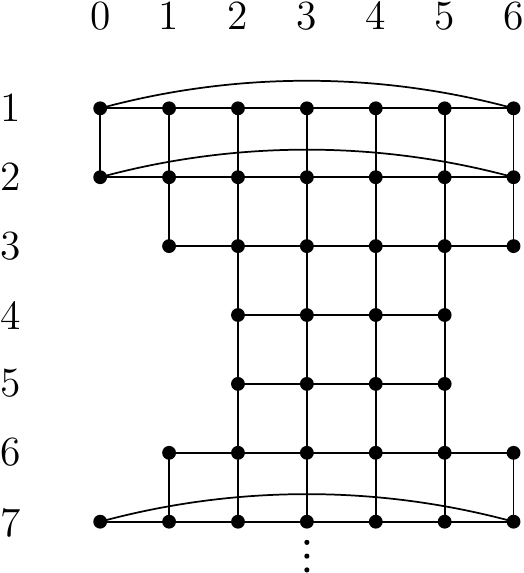}
\end{center}
\caption{The residue graph for edges removable from $P_m\times C_7$.}
\label{fig:pmc7}
\end{figure}

\begin{proof}[Proof of \ref{thm:periods}.d).]
In $P_m\times C_3$ each edge $\{(3,i),(4,i)\}$ is removable for $i=0,1,2$ because each residue graph contains a configuration of type C. As before, this implies an equivalence
\begin{equation*}
\ind(P_m\times C_3)\htpyequiv \ind(P_{m-3}\times C_{3})\ast\ind(P_3\times C_3)\htpyequiv\ind(P_{m-3}\times C_{3})\ast S^1= \susp^2\,\ind(P_{m-3}\times C_3)
\end{equation*}
\end{proof}

\begin{proof}[Proof of \ref{thm:periods}.e).]
In $P_m\times C_5$ each edge $\{(2,i),(3,i)\}$ is removable for $i=0,1,2,3,4$ because each residue graph contains a configuration of type A with vertices $(1,i-1)$ and $(1,i+1)$ having degree $1$. Again, this means
\begin{equation*}
\ind(P_m\times C_5)\htpyequiv \ind(P_{m-2}\times C_{5})\ast\ind(P_2\times C_5)\htpyequiv\ind(P_{m-2}\times C_{5})\ast S^1= \susp^2\,\ind(P_{m-2}\times C_5)
\end{equation*}
\end{proof}

\begin{proof}[Proof of \ref{thm:periods}.f).]
We want to show that the edges $e_i=\{(4,i),(5,i)\}$, $i=0,\ldots,6$ are sequentially removable. Then the result will follow as before:
\begin{equation*}
\ind(P_m\times C_7)\htpyequiv \ind(P_{m-4}\times C_{7})\ast\ind(P_4\times C_7)\htpyequiv\ind(P_{m-4}\times C_{7})\ast S^5= \susp^6\,\ind(P_{m-4}\times C_7).
\end{equation*}
The residue graph of $e_0$ is the graph $G$ from Fig.\ref{fig:pmc7}. We need to show that the independence complex of that graph is contractible. To do this, we will first show that each of the vertices $(4,j)$, $j=2,3,4,5$, is removable in $G$. By symmetry, it suffices to consider $(4,2)$ and $(4,3)$.

Consider first the vertex $(4,2)$ and its residue graph $G\setminus N[(4,2)]$. It can be transformed in the following steps: remove $N[(2,1)]$ (by \ref{lemma:simple2}.b with $u=(3,1)$); remove $N[(1,6)]$ (by \ref{lemma:simple2}.b with $u=(1,0)$); remove $N[(1,3)]$ (by \ref{lemma:simple2}.b with $u=(1,2)$); remove $N[(3,4)]$ (by \ref{lemma:simple2}.b with $u=(3,3)$). In the final graph $(2,5)$ is isolated.

Now we prove the residue graph $G\setminus N[(4,3)]$ has a contractible independence complex. Decompose it as follows: remove $(3,1),(3,2),(2,1),(2,2)$ (by \ref{lemma:simple2}.c); remove $(1,6)$ (by \ref{lemma:simple2}.a with $u=(2,0)$); remove $(2,5)$ (by \ref{lemma:simple2}.a with $u=(3,6)$); remove $N[(1,4)]$ (by \ref{lemma:simple2}.b with $u=(1,5)$). In the final graph $(2,3)$ is isolated.

Since all the vertices in row $4$ of $G$ are removable, $\ind(G)$ is homotopy equivalent to the join $\ind(G[1,2,3])\ast \ind(G[5,\ldots])$, where $G[\ldots]$ means the subgraph of $G$ spanned by the numbered rows. But a direct calculation shows that $\ind(G[1,2,3])$ is contractible, hence so is $\ind(G)$. This ends the proof that the edge $e_0=\{(4,0),(5,0)\}$ of $P_m\times C_7$ was removable.

For all other edges $e_i$ in a sequence the residue graph will look exactly like $G$ with possibly some edges between rows $4$ and $5$ missing. This has no impact on contractibility since all of the above proof took part in rows $1,2,3$ of Fig.\ref{fig:pmc7}. That means that all $e_i$ are removable, as required.
\end{proof}

\begin{proof}[Proof of Proposition \ref{prop:other}.]
We just sketch the arguments and the reader can check the details. Part a) is a result of \cite{Koz} and also follows from observing that the edge $\{3,4\}$ of the path $P_n$ is removable as its residue graph has an isolated vertex $1$. Part b) follows directly from Lemma \ref{lemma:simple2}.c. 

For part c), each edge $\{(i,4),(i,5)\}$, $i=1,2,3$ of $P_3\times P_n$ is removable because their residue graphs either contain, or can easily be reduced to contain, a configuration of type C or D. Then the graph splits into two components and we conclude as usually. 

In d) we first show that each edge $\{(i,0),(i,4)\}$, $i=0,1,2$ is insertable into $C_3\times C_n$ because the residue graph contains a configuration of type C. Then in the enlarged graph the obvious edges which must be removed to obtain a disjoint union $C_3\times C_{n-3}\sqcup C_3\times P_3$ are indeed removable, again because of a type C configuration in their residue graphs. We conclude as always.
\end{proof}

\section{Cylinders with even circumference: Patterns}
\label{sect:patterns}

To prove the results about cylinders of even circumference we will need quite a lot of notation. On the plus side, once all the objects are properly defined, the proofs will follow in a fairly straightforward way. It is perhaps instructive to read this and the following sections simultaneously. The next section contains a working example of what is going on for $n=6$. From now on $n$, the length of the cycle, is a fixed even integer which will not appear in the notation.

\empha{A pattern} $\nicep$ is a matrix of size $2\times n$ with $0/1$ entries and such that if $\nicep(1,i)=1$ then $\nicep(2,i)=1$, i.e. below a $1$ in the first row there is always another $1$ in the second row. An example of a pattern is
$$
\nicep=\left(\begin{array}{cccccc}1&0&1&0&0&0\\ 1&1&1&1&0&1\end{array}\right).
$$
We also call $n$ the length of the pattern. The rows of a pattern are indexed by $1$ and $2$, while the columns are indexed with $0,\ldots,n-1$, as the vertices of $C_n$. Given $i$ we say $\nicep(1,i)$ is `above' $\nicep(2,i)$ and $\nicep(2,i)$ is `below' $\nicep(1,i)$. We identify a pattern with patterns obtained by a cyclic shift or by a reflection, since they will define isomorphic graphs (see below). Also the words `left', `right' and `adjacent' are understood in the cyclic sense.

Given a pattern $\nicep$ define $G(\nicep;m)$ as the induced subgraph of $P_m\times C_n$ obtained by removing those vertices $(1,i)$ and $(2,i)$ for which $\nicep(1,i)=0$, resp. $\nicep(2,i)=0$. This amounts to applying a `bit mask' defined by $\nicep$ to the first two rows of $P_m\times C_n$. The graph $G(\nicep;m)$ for the pattern $\nicep$ above is:
\begin{center}
\includegraphics[scale=0.85]{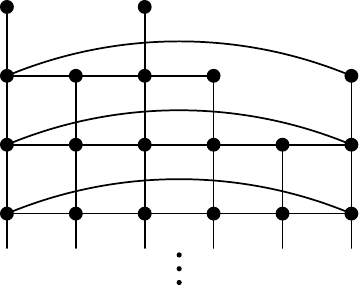}
\end{center}

Define the simplified notation $Z(\nicep;m):=Z(G(\nicep;m))$. If $\nicei$ denotes the all-ones pattern then $Z(\nicei;m)=Z(P_m\times C_n)$ is the value we are interested in.

We now need names for some structures within a row:
\begin{itemize}
\item A \empha{singleton} is a single $1$ with a $0$ both on the left and on the right. 
\item A \empha{block} is a contiguous sequence of $1$s of length at least $3$, which is bounded by a $0$ both on the left and on the right.
\item A \empha{run} is a sequence of blocks and singletons separated by single $0$s.
\item A \empha{nice run} is a run in which every block has length exactly $3$.
\end{itemize}

\begin{example}
The sequence $01010111010111010$ is a nice run. The sequence $01110111101010$ is a run, but it is not a nice run. The sequences $01011010$ and $0101001110$ are not runs.
\end{example}

A pattern is \empha{reducible} if the only occurrences of $1$ in the first row are singletons. Otherwise we call it \empha{irreducible}. In particular, a pattern whose first row contains only $0$s is reducible.

We now need three types of pattern transformations, which we denote $V$ (for vertex), $N$ (for neighbourhood) and $R$ (for reduction). The first two can be performed on any pattern. Suppose $\nicep$ is a pattern and $i$ is an index such that $\nicep(1,i)=1$.
\begin{itemize}
\item An \empha{operation of type $V$} sets $\nicep(1,i)=0$. We denote the resulting pattern $\nicep^{V,i}$.
\item An \empha{operation of type $N$} sets $\nicep(1,i-1)=\nicep(1,i)=\nicep(1,i+1)=\nicep(2,i)=0$. We denote the resulting pattern $\nicep^{N,i}$.
\end{itemize}
All other entries of the pattern remain unchanged. If it is clear what index $i$ is used we will abbreviate the notation to $\nicep^V$ and $\nicep^N$.
\begin{lemma}
\label{lem:zvn}
For any pattern $\nicep$ and index $i$ such that $\nicep(1,i)=1$ we have
$$Z(\nicep;m)=Z(\nicep^{V,i};m)-Z(\nicep^{N,i};m),\quad m\geq 2.$$
\end{lemma}
\begin{proof}
This is exactly the first equality of Lemma \ref{lemma:zadditive} for $G=G(\nicep;m)$.
\end{proof}
The third \empha{operation, of type $R$}, can be applied to a reducible pattern $\nicep$ and works as follows. First, temporarily extend $\nicep$ with a new, third row, filled with ones. Now for every index $i$ such that $\nicep(1,i)=1$ (note that no two such $i$ are adjacent) make an assignment
$$\nicep(1,i)=\nicep(2,i-1)=\nicep(2,i)=\nicep(2,i+1)=\nicep(3,i)=0.$$
Having done this for all such $i$ remove the first row (which is now all zeroes) and let $\nicep^R$ be the pattern formed by the second and third row.
\begin{lemma}
\label{lem:zr}
Suppose $\nicep$ is a reducible pattern and $k$ is the number of ones in its first row. Then
$$Z(\nicep;m)=(-1)^kZ(\nicep^R;m-1),\quad m\geq 2.$$
\end{lemma}
\begin{proof}
It follows from a $k$-fold application of Lemma \ref{lemma:simple2}.b).
\end{proof}

We now describe a class of patterns which arise when one applies the above operations in some specific way to the all-ones pattern. A pattern $\nicep$ is called \empha{proper} if it satisfies the following conditions:
\begin{itemize}
\item The whole second row is either a run or it consists of only $1$s.
\item If the second row has only $1$s then the first row is a nice run.
\item Above every singleton $1$ in the second row there is a $0$ in the first row.
\item Above every block of length $3$ in the second row there are $0$s in the first row.
\item If $B$ is any block in the second row of length at least $4$ then above $B$ there is a nice run $R$, subject to the conditions:
\begin{itemize}
\item if the leftmost group of $1$s in $R$ is a block (of length $3$) then the leftmost $1$ of that block is located exactly above the $3$rd position of $B$,
\item if the leftmost group of $1$s in $R$ is a singleton then it is located exactly above the $2$nd or $3$rd position of $B$.
\end{itemize}
By symmetry the same rules apply to the rightmost end of $B$ and $R$.
\end{itemize}

Note that a proper pattern does not contain the sequences $0110$ nor $1001$ in any row. Also note that the first row of any proper pattern can only contain singletons and blocks of length $3$, and no other groups of $1$s.

\begin{example}
\label{ex:proper}
Here are some proper patterns:
\begin{eqnarray*}
\mathcal{A}=\left(\begin{array}{cccccccccc}0&1&0&1&0&0&0&0&0&0\\1&1&1&1&1&0&1&1&1&0 \end{array}\right) &
\mathcal{B}=\left(\begin{array}{cccccccccc}1&0&1&1&1&0&1&1&1&0\\1&1&1&1&1&1&1&1&1&1 \end{array}\right) \\
\mathcal{C}=\left(\begin{array}{cccccccccc}0&0&1&1&1&0&0&0&0&0\\1&1&1&1&1&1&1&0&1&0 \end{array}\right) &
\mathcal{D}=\left(\begin{array}{cccccccccc}0&0&1&0&1&1&1&0&0&0\\1&1&1&1&1&1&1&1&1&0 \end{array}\right)
\end{eqnarray*}

\end{example}

A pattern is called \empha{initial} if it is obtained from the all-ones pattern $\nicei$ by performing, for each even index $i=0,2,4,\ldots,n-2$ one of the operations of type $V$ or $N$. It means there should be $2^{n/2}$ initial patterns, but some of them can be identified via cyclic shift or reflection. One can easily see that every initial pattern is reducible. Moreover, by a repeated application of Lemma \ref{lem:zvn} we get that $Z(P_m\times C_n)$ is a linear combination of the numbers $Z(\nicep;m)$ for initial patterns $\nicep$. More importantly, we have:

\begin{lemma}
\label{lem:initialproper}
Every initial pattern is proper.
\end{lemma}
\begin{proof}
If the operations we perform in positions $i=0,2,\ldots,n-2$ are all $V$ or all $N$ then we get one of the patterns described in Lemma \ref{lem:mu0patterns}. If we perform $N$ in points $i$, $i+2$ we get a singleton in position $i+1$ of the second row with a $0$ above it. For a choice of $NVN$ in $i$, $i+2$, $i+4$ we get a block of length $3$ in the second row with $0$s above. Finally for a longer segment $NV\cdots VN$ the outcome is a block with a nice run of singletons starting and ending above the $3$rd position in the block. We can never get two adjacent $0$s in the second row, so it is a run. Here is a summary of the possible outcomes:
$$
\bordermatrix{&&&\mathbf{N}&&V&&V&&V&&V&&\mathbf{N}&&V&&\mathbf{N}&&\mathbf{N}&&\cr&\cdots&0&0&0&0&1&0&1&0&1&0&0&0&0&0&0&0&0&0&0&\cdots\cr&\cdots&1&0&1&1&1&1&1&1&1&1&1&0&1&1&1&0&1&0&1&\cdots}
$$
where the labels $V,N$ indicate which operation was applied at a position.
\end{proof}

Now we introduce the main tool: an invariant which splits proper patterns into classes which can be analyzed recursively.

\begin{definition}
\label{def:megadefinition}
For any proper pattern $\nicep$ we define the $\mu$-invariant as follows:
\begin{eqnarray*}
\mu(\nicep)&=&\mathrm{(\ number\ of\ blocks\ in\ the\ first\ row\ of\ }\nicep\mathrm{\ )} + \\
& & \mathrm{(\ number\ of\ blocks\ in\ the\ second\ row\ of\ }\nicep\mathrm{\ )}.
\end{eqnarray*}
\end{definition}

\begin{example}
All the proper patterns in Example \ref{ex:proper} have $\mu$-invariant $2$.
\end{example}

\begin{lemma}
\label{lem:mu0patterns}
For every proper pattern $\nicep$ we have $0\leq\mu(\nicep)\leq n/4$. The only patterns with $\mu(\nicep)=0$ are
$$
\nicep_1=\left(\begin{array}{ccccccc}1&0&1&0&\cdots&1&0\\ 1&1&1&1&\cdots&1&1\end{array}\right),\quad \nicep_2=\left(\begin{array}{ccccccc}0&0&0&0&\cdots&0&0\\ 1&0&1&0&\cdots&1&0\end{array}\right).
$$
\end{lemma}
\begin{proof}
If $\mu(\nicep)=0$ then $\nicep$ has no blocks in either row. If the second row is all-ones then the first row must be a nice run with no blocks, so $\nicep=\nicep_1$. Otherwise the second row is an alternating $0/1$ but then the first row cannot have a $1$ anywhere, so $\nicep=\nicep_2$.

Now consider the following map. To every block in the second row we associate its two rightmost points, its leftmost point and the immediate left neighbour of the leftmost point. To every block in the first row we associate its three points and the point immediately left. This way every block which contributes to $\mu(\nicep)$ is given $4$ points, and those sets are disjoint for different blocks. The only non-obvious part of the last claim follows since a block in  the first row does not have any point over the two outermost points of the block below it. That ends the proof of the upper bound. 
\end{proof}

Next come the crucial observations about operations on proper patterns and their $\mu$-invariants.

\begin{proposition}
\label{prop:propervn}
Suppose $\nicep$ is a proper, irreducible pattern and let $i$ be the index of the middle element of some block in the first row. Then $\nicep^{V,i}$ and $\nicep^{N,i}$ are proper and
$$\mu(\nicep^{V,i})=\mu(\nicep)-1,\quad \mu(\nicep^{N,i})=\mu(\nicep).$$
\end{proposition}
\begin{proof}
There are two cases, depending on whether the block of length $3$ centered at $i$ is, or is not, the outermost group of $1$s in its run. If it is the outermost one then it starts over the $3$rd element of the block below it. The two possible situations are depicted below. 
$$
\left(\begin{array}{ccccccccc}\cdots&0&0&0&1&1&1&0&\cdots \\ \cdots&0&1&1&1&1&1&1&\cdots\end{array}\right),\quad
\left(\begin{array}{ccccccccc}\cdots&1&0&1&1&1&0&1&\cdots \\ \cdots&1&1&1&1&1&1&1&\cdots\end{array}\right).
$$
In $\nicep^V$ the second row is the same as in $\nicep$. Above the current block we still have a nice run and its outermost $1$s are in the same positions. That means $\nicep^V$ is still proper. The number of blocks in the first row dropped by one, so $\mu(\nicep^V)=\mu(\nicep)-1$.

The proof for $\nicep^N$ depends on the two cases. In the first case an operation of type $N$ splits the block in the second row creating a new block of size $3$ with $0$s above it. In the second block that comes out of the splitting the first two $1$s have $0$s above them, so whatever run there was in $\nicep$ it is still there and starts in an allowed position. That means we get a proper pattern. One block was removed and one split into two, so $\mu$ does not change.

In the second case the situation is similar. We increase the number of blocks in the second row by one while removing one block from the first row. The two outermost positions in the new block(s) have $0$s above them, so the nice runs which remain above them start in correct positions. Again $\nicep^N$ is proper.
\end{proof}

\begin{proposition}
\label{prop:properr}
If $\nicep$ is a proper, reducible pattern then $\nicep^{R}$ is proper and
$$\mu(\nicep^{R})=\mu(\nicep).$$
\end{proposition}
\begin{proof}
Consider first the case when $\nicep$ has only $0$s in the first row. Then the second row is a run with blocks only of size $3$ (because a longer block would require something above it). This means $\nicep^R$ has a full second row with a nice run in the first row. Such pattern is proper and $\mu(\nicep^R)=\mu(\nicep)$ as we count the same blocks.

Now we move to the case when $\nicep$ has at least one $1$ in the first row. Note that $\nicep(1,i)=1$ if and only if $\nicep^R(2,i)=0$. That last condition implies that the second row of $\nicep^R$ is a run (as the first row of $\nicep$ does not contain the sequences $11$ nor $1001$).

If $\nicep^R(2,i)$ is a singleton $1$ then in $\nicep$ we must have had $\nicep(1,i-1)=\nicep(1,i+1)=1$ but then $\nicep(2,i)$ was erased by the operation $R$ and therefore $\nicep^R(1,i)=0$. This proves $\nicep^R$ has zeroes above singletones of the second row.

Now consider any block $B$ in the second row of $\nicep^R$ and assume without loss of generality that it occupies positions $1,\ldots,l$, hence $\nicep(1,0)=\nicep(1,l+1)=1$, $\nicep^R(2,0)=\nicep^R(2,l+1)=0$ and $\nicep(1,i)=0$ for all $1\leq i\leq l$. It means that the situation in $\nicep$ must have looked like one of these (up to symmetry):

\[
\begin{array}{ccccccccccccc}
                   &   & 0 &   &   &   &      &   &   &   &l+1&   &\\
\ldelim({2}{0.5em} & 0 & 1 & 0 & 0 & 0 &      & 0 & 0 & 0 & 1 & 0 & \rdelim){2}{0.5em} \\
                   & 1 & 1 & 1 & 0 & 1 &\cdots& 1 & 0 & 1 & 1 & 1 &\\
                   &   & 0 & B & B & B &      & B & B & B & 0 &   &\\
\end{array}
\]

\[
\begin{array}{ccccccccccccc}
                   &   & 0 &   &   &   &      &   &   &   &l+1&   &\\
\ldelim({2}{0.5em} & 0 & 1 & 0 & 0 & 0 &      & 0 & 0 & 0 & 1 & 0 & \rdelim){2}{0.5em} \\
                   & 1 & 1 & 1 & 0 & 1 &\cdots& 0 & 1 & 1 & 1 & 1 &\\
                   &   & 0 & B & B & B &      & B & B & B & 0 &   &\\
\end{array}
\]

\[
\begin{array}{ccccccccccccc}
                   &   & 0 &   &   &   &      &   &   &   &l+1&   &\\
\ldelim({2}{0.5em} & 0 & 1 & 0 & 0 & 0 &      & 0 & 0 & 0 & 1 & 0 & \rdelim){2}{0.5em} \\
                   & 1 & 1 & 1 & 1 & 0 &\cdots& 0 & 1 & 1 & 1 & 1 &\\
                   &   & 0 & B & B & B &      & B & B & B & 0 &   &\\
\end{array}
\]

The letters $B$ indicate where the block $B$ will stretch in what will become the future second row of $\nicep^R$. The $1$s in $\nicep(1,0)$ and $\nicep(1,l+1)$ must be located above the $2$nd or $3$rd element of a block. The part of the second row in $\nicep$ denoted by $\cdots$ is a run with no $1$s above, so it must be a nice run. It follows that in $\nicep^R$ above $B$ we will get a nice run and by checking the three cases we see that the run starts above the $2$nd or $3$rd element of $B$, and it only starts above the $2$nd element if it has a singleton there. 

There is just one way in which we can obtain a block $B$ of size $3$:
\[
\begin{array}{ccccccc}
                   & 0 &   &   &   &l+1&\\
\ldelim({2}{0.5em} & 1 & 0 & 0 & 0 & 1 &\rdelim){2}{0.5em} \\
                   & 1 & 1 & 0 & 1 & 1 &\\
                   & 0 & B & B & B & 0 &\\
\end{array}
\]
and then the operation $R$ will erase everything above that block. This completes the check that the pattern $\nicep^R$ is proper.

It remains to compute $\mu(\nicep^R)$. Our previous discussion implies that:
\begin{itemize}
\item every block of size $3$ in the second row of $\nicep$ becomes a block in the first row of $\nicep^R$,
\item every two consecutive blocks longer than $3$ yield, between them, a block $B$ in the second row of $\nicep^R$,
\end{itemize}
and every block in $\nicep^R$ arises in this way. It means that every block in $\nicep$ contributes one to the count of blocks in $\nicep^R$ (in either first or second row). That proves $\mu(\nicep^R)=\mu(\nicep)$.
\end{proof}

\section{An example: $n=6$}
\label{sect:example6}

First of all the value $Z(P_m\times C_6)=Z(\nicei;m)$ for the all-ones pattern $\nicei$ splits into a linear combination of $Z$-values for the following patterns.
\begin{eqnarray*}
\mathcal{A}=\bordermatrix{&V&&V&&V&\cr&0&1&0&1&0&1\cr&1&1&1&1&1&1} & \mathcal{B}=\bordermatrix{&V&&V&&N&\cr&0&1&0&0&0&0\cr&1&1&1&1&0&1}\\
\mathcal{C}=\bordermatrix{&V&&N&&N&\cr&0&0&0&0&0&0\cr&1&1&0&1&0&1} & \mathcal{D}=\bordermatrix{&N&&N&&N&\cr&0&0&0&0&0&0\cr&0&1&0&1&0&1}
\end{eqnarray*}
The labels $V,N$ indicate which operation was applied to the particular position $i=0,2,4$. Any other pattern we get is isomorphic to one of these, and Lemma \ref{lem:zvn} unfolds recursively into:
\begin{equation}
\label{eq:6split}
Z(\nicei;m)=Z(\mathcal{A};m)-3Z(\mathcal{B};m)+3Z(\mathcal{C};m)-Z(\mathcal{D};m).
\end{equation}
We also see that (Definition \ref{def:megadefinition}):
$$\mu(\mathcal{A})=0+0=0,\ \mu(\mathcal{B})=0+1=1,\ \mu(\mathcal{C})=0+1=1,\ \mu(\mathcal{D})=0+0=0.$$
All the patterns we have now are reducible. The first obvious reductions are
$$\mathcal{A}^R=\mathcal{D},\quad \mathcal{D}^R=\mathcal{A}$$
and by Lemma \ref{lem:zr} they lead to
$$Z(\mathcal{D};m)=Z(\mathcal{A};m-1),\quad Z(\mathcal{A};m)=-Z(\mathcal{D};m-1)=-Z(\mathcal{A};m-2).$$
It means that given the initial conditions for $Z(\mathcal{A};m)$ and $Z(\mathcal{D};m)$ we have now completely determined those sequences and their generating functions. 

Note that $\mathcal{A}$ and $\mathcal{D}$ had $\mu$-invariant $0$. Now we move on to the patterns with the next $\mu$-invariant value $1$. We can reduce $\mathcal{B}$ (even three times) and $\mathcal{C}$ and apply Lemma \ref{lem:zr}:
$$\mathcal{B}^{RRR}=\mathcal{E},\ \mathcal{C}^R=\mathcal{E}; \quad Z(\mathcal{B};m)=-Z(\mathcal{E};m-3),\ Z(\mathcal{C};m)=Z(\mathcal{E};m-1)$$
where
$$
\mathcal{E}=\left(\begin{array}{cccccc}1&0&1&1&1&0\\1&1&1&1&1&1\end{array}\right).
$$
Still $\mu(\mathcal{E})=1$. Now we can apply a $V,N$-type splitting in the middle of the length $3$ block in the first row of $\mathcal{E}$, as in Proposition \ref{prop:propervn}. We have by Lemma \ref{lem:zvn}:
$$\mathcal{E}^V=\mathcal{A},\ \mathcal{E}^N=\mathcal{B}; \quad \quad Z(\mathcal{E};m)=Z(\mathcal{A};m)- Z(\mathcal{B};m)=Z(\mathcal{A};m)+ Z(\mathcal{E};m-3)$$
where $\mu(\mathcal{A})=0$, so the sequence $Z(\mathcal{A};m)$ is already known.

This recursively determines all the sequences and it is a matter of a mechanical calculation to derive their generating functions (some care must be given to the initial conditions). We can also check periodicities directly. The sequences with $\mu$-invariant $0$ are $4$-periodic:
$$Z(\mathcal{A};m)=-Z(\mathcal{A};m-2)=Z(\mathcal{A};m-4)$$
and those with $\mu$-invariant $1$ are $12$-periodic:
$$Z(\mathcal{E};m)=Z(\mathcal{A};m)+Z(\mathcal{A};m-3)+Z(\mathcal{A};m-6)+Z(\mathcal{A};m-9)+Z(\mathcal{E};m-12)=Z(\mathcal{E};m-12)$$
since $Z(\mathcal{A};m)=-Z(\mathcal{A};m-6)$. By (\ref{eq:6split}) this means $12$-periodicity of $Z(P_m\times C_6)$.

\section{Proof of the Theorem~\ref{thm:genfun}}
\label{sect:weakproof}
Everything is now ready to prove Theorem \ref{thm:genfun}. We are going to deduce it from a refined version given below. Throughout this section an even number $n$ is still fixed. For any pattern $\nicep$ of length $n$ define the generating function
\begin{equation}
\label{eqn:fp}
f_\nicep(t)=\sum_{m=0}^\infty Z(\nicep;m)t^m.
\end{equation}
\begin{proposition}
\label{prop:weakrefined}
For any proper pattern $\nicep$ the function $f_\nicep(t)$ is a rational function such that all zeroes of its denominator are complex roots of unity.
\end{proposition}

The sequence $Z(P_m\times C_n)$ is a linear combination of sequences $Z(\nicep;m)$ for initial patterns $\nicep$ (modulo initial conditions). Every initial pattern is proper (Lemma \ref{lem:initialproper}), so Theorem \ref{thm:genfun} follows. It remains to prove Proposition \ref{prop:weakrefined}, and this is done along the lines of the example in Section \ref{sect:example6}.

\begin{proof}
We will prove the statement by induction on the $\mu$-invariant of $\nicep$. If $\mu(\nicep)=0$, then $\nicep$ is one of the patterns from Lemma \ref{lem:mu0patterns}. Each of them satisfies $\nicep^{RR}=\nicep$ and by Lemma \ref{lem:zr}:
$$Z(\nicep;m)=(-1)^{n/2}Z(\nicep;m-2)$$
hence $f_\nicep(t)$ has the form
$$\frac{a+bt}{1-(-1)^{n/2}t^2}.$$

Now consider any fixed value $\mu>0$ of the $\mu$-invariant and suppose the result was proved for all proper patterns with smaller $\mu$-invariants. Consider the directed graph whose vertices are all proper patterns with that invariant $\mu$. For any reducible $\nicep$ there is an edge $\nicep\to\nicep^R$ and for any irreducible $\nicep$ there is an edge $\nicep\to\nicep^N$ for some (only one) choice of $N$-type operation in the middle of a block. Since the graph is finite and the outdegree of each vertex is $1$, it consists of directed cycles with some attached trees pointing towards the cycles.

If $\nicep$ is a vertex on one of the cycles, then by moving along the cycle and performing the operations prescribed by the edges we will get back to $\nicep$ and obtain (Lemmas \ref{lem:zvn}, \ref{lem:zr} and Propositions \ref{prop:propervn},\ref{prop:properr}) a recursive equation of the form
\begin{equation}\label{recursiveeq}Z(\nicep;m)=\pm Z(\nicep;m-a)+\sum_\mathcal{R}\pm Z(\mathcal{R};m-b_\mathcal{R})\end{equation}
where $a>0$, $b_\mathcal{R}\geq 0$ and $\mathcal{R}$ runs through some proper patterns with invariant $\mu-1$ (Proposition \ref{prop:propervn}). Now the result follows by induction (the generating function $f_\nicep(t)$ will add an extra factor $1\pm t^a$ to the denominator coming from the combination of functions $f_\mathcal{R}(t)$).

If $\nicep$ is not on a cycle then it has a path to a cycle and the result follows in the same way.
\end{proof}

\begin{remark}
It is also clear that in order to prove Conjecture~\ref{thm:genfunmedium} one must have better control over the cycle lengths in the directed graph appearing in the proof. We will construct a more accessible model for this in the next section, see Theorem~\ref{thm:necklace-to-patterns}.
\end{remark}

\section{Necklaces}
\label{section:neck}

In this section we describe an appealing combinatorial model which encodes the reducibility relation between patterns. As before $n$ is an even positive integer and $k$ is any positive integer.

We define a \textbf{\emph{$(k,n)$-necklace}}. It is a collection of $2k$ points (\emph{stones}) distributed along the circumference of a circle of length $n$, together with an assignment of a number from $\{-2,-1,1,2\}$ to each of the stones. We call these numbers \emph{stone vectors} and we think of them as actual short vectors attached to the stones and tangent to the circle. The vector points $1$ or $2$ units clockwise (positive value) or anti-clockwise (negative value) from each stone and we say a stone \emph{faces} the direction of its vector. See Fig.\ref{fig:neck1} for an example worth more than a thousand words.

During a \emph{jump} a stone moves $1$ or $2$ units along the circle in the direction and distance prescribed by its vector. The configuration of stones and vectors is subject to the following conditions:
\begin{itemize}
\item consecutive stones face in opposite directions,
\item if two consecutive stones face away from each other then their distance is an odd integer,
\item if two consecutive stones face towards each other then their distance \emph{plus} the lengths of their vectors is an odd integer,
\item if two consecutive stones face towards each other then their distance is at least $3$; moreover if their distance is exactly $3$ then their vectors have length $1$.
\end{itemize}
The last two conditions can be conveniently rephrased as follows: if two stones facing towards each other simultaneously jump then after the jump their distance will be an odd integer and they will not land in the same point nor jump over one another.

We identify $(k,n)$-necklaces which differ by an isometry of the circle. Clearly the number of $(k,n)$-necklaces is finite.

\begin{figure}
\begin{center}
\begin{tabular}{cc}
\includegraphics{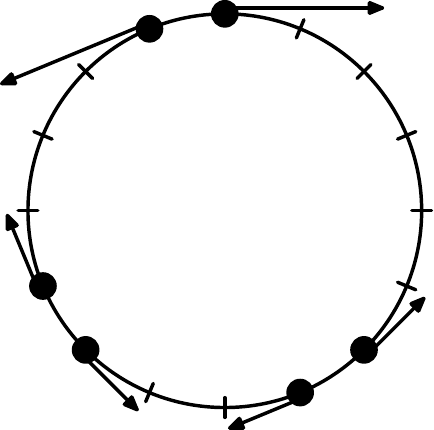} & \includegraphics{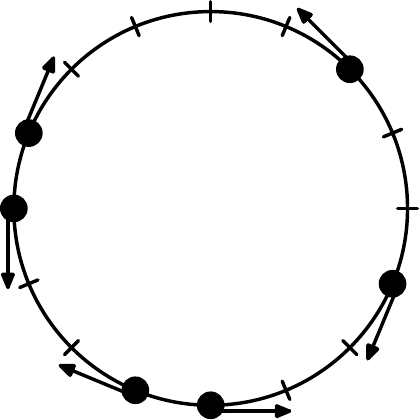}\\
a) & b)
\end{tabular}
\caption{\label{fig:neck1}Two sample $(3,16)$-necklaces. Each has $6$ stones. The arrows are stone vectors of length $1$ (shorter) or $2$ (longer). The necklace in b) is the image under $T$ of the necklace in a).}
\end{center}
\end{figure}

Next we describe a \textbf{necklace transformation $T$} which takes a $(k,n)$-necklace and performs the following operations:
\begin{itemize}
\item (JUMP) all stones jump as dictated by their vectors,
\item (TURN) all stone vectors change according to the rule
$$-2\to 1, \quad -1\to 2, \quad 1\to -2, \quad 2\to -1,$$
i.e. both direction and length are switched to the other option,
\item (FIX) if any two stones find themselves in distance $3$ facing each other and any of their vectors has length $2$, then adjust the offending vectors by reducing their length to $1$.
\end{itemize}
An example of $N$ and $TN$ is shown in Fig.\ref{fig:neck1}. It is easy to check that if $N$ is a $(k,n)$-necklace then so is $TN$. 

\begin{definition}
\label{def:necklace-graph}
Define $\neck(k,n)$ to be the directed graph whose vertices are all the isomorphism classes of $(k,n)$-necklaces and such that for each $(k,n)$-necklace $N$ there is a directed edge $N\to TN$.
\end{definition}

\begin{lemma}
\label{lemma:neck-general}
The graph $\neck(k,n)$ is nonempty if and only if $1\leq k\leq n/4$ and it is always a disjoint union of directed cycles.
\end{lemma}
\begin{proof}
To each stone whose vector faces clockwise we associate the open arc segment of length $2$ from that stone in the direction of its vector. To each stone whose vector faces counter-clockwise we associate the open arc segment of length $2$ of which this stone is the midpoint. The segments associated to different stones are disjoint hence $2k\cdot 2\leq n$, as required (compare the proof of Lemma \ref{lem:mu0patterns}). 

The out-degree of every vertex in $\neck(k,n)$ is $1$, so it suffices to show that the in-degree is at least $1$. Given a $(k,n)$-necklace $N$ let $T^{-1}$ be the following operations:
\begin{itemize}
\item for each stone which \emph{does not} face towards another stone in distance $3$, change the stone vector according to the rule
$$-2\to -1, \quad -1\to -2, \quad 1\to 2, \quad 2\to 1,$$
\item jump with all the stones,
\item change all stone vectors according to the rule
$$-2\to 2, \quad -1\to 1, \quad 1\to -1, \quad 2\to -2.$$
\end{itemize}
One easily checks that $T^{-1}N$ is a $(k,n)$-necklace and that $TT^{-1}N=T^{-1}TN=N$.
\end{proof}

Some boundary cases of $\neck(k,n)$ are easy to work out.
\begin{lemma}
\label{lemma:neck-small}
For any even $n$ the graph $\neck(1,n)$ is a cycle of length $n-3$. For any $k$ the graph $\neck(k,4k)$ is a single vertex with a loop. For any $k$ the graph $\neck(k,4k+2)$ is a cycle of length $k+2$ and $\lfloor k/2 \rfloor$ isolated vertices with loops.
\end{lemma}
\begin{proof}
A $(1,n)$-necklace is determined by a choice of an odd number $3\leq d\leq n-1$ (the length of the arc along which the two stones face each other) and a choice of $\epsilon\in\{1,2\}$ (the length of the vectors at both stones which must be the same due to the parity constraints), with the restriction that if $d=3$ then $\epsilon=1$. If we denote the resulting necklace $A_{n,d}^\epsilon$ then
$$TA_{n,n-1}^1=A_{n,3}^1, \quad TA_{n,d}^1=A_{n,n-d+2}^2 \ (3\leq d\leq n-3), \quad TA_{n,d}^2=A_{n,n-d+4}^1 \ (5\leq d\leq n-1)$$
and it is easy to check that they assemble into a $(n-3)$-cycle.

An argument as in Lemma \ref{lemma:neck-general} shows that there is just one $(k,4k)$-necklace $N$, with distances between stones alternating between $3$ (stones facing each other) and $1$ (stones facing away) and all vectors of length $1$. It satisfies $TN=N$.

The analysis of the last case again requires the enumeration of all possible cases and we leave it to the interested reader.
\end{proof}

The following is our main conjecture about $\neck(k,n)$.

\begin{conjecture}
\label{con:cyc-len}
The length of every cycle in the graph $\neck(k,n)$ divides $n-3k$. In other words, for every $(k,n)$-necklace $N$ we have $T^{n-3k}N=N$.
\end{conjecture}

This conjecture was experimentally verified for all even $n\leq 36$, see Table 3 in the Appendix. 

It is now time to explain what necklaces have to do with patterns and what Conjecture \ref{con:cyc-len} has to do with Conjecture \ref{thm:genfunmedium}.

Intuitively, $(k,n)$-necklaces are meant to correspond to reducible proper patterns $\nicep$ of length $n$ and $\mu(\nicep)=k$. The operation $T$ mimics the reduction $\nicep\to\nicep^R$, although the details of this correspondence are a bit more complicated (see proof of Theorem \ref{thm:necklace-to-patterns}). The lengths of cycles in the necklace graph $\neck(k,n)$ determine the constants $a$ in the recursive equations (\ref{recursiveeq}) and therefore also the exponents in the denominators of $f_n(t)$ (Conjecture \ref{thm:genfunmedium}). A precise statement is the following.

\begin{figure}
\begin{center}
\begin{tabular}{cc}
\includegraphics{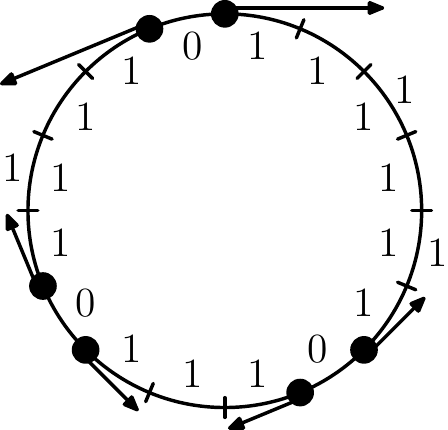} & \includegraphics{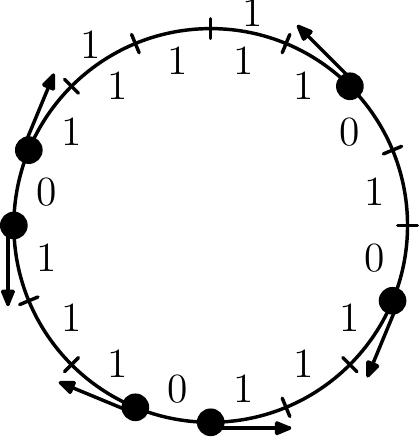}\\
a) & b)
\end{tabular}
\caption{\label{fig:neck2} The correspondence between necklaces and patterns. If $N$ is the necklace then the numbers inside the circle form the second row of the pattern $UN$ and the numbers outside form its first row (only $1$'s are shown in the first row, the remaining entries are $0$'s). The second row has a block between each pair of stones facing each other. Over each block of length greater than $3$ the number of outermost $0$'s in the first row equals the length of the stone vector.}
\end{center}
\end{figure}

\begin{theorem}
\label{thm:necklace-to-patterns}
Let $g_{i,n}$ be the any common multiple of the lengths of all cycles in the graph $\neck(i,n)$. Suppose $\nicep$ is a proper pattern of even length $n$ and $\mu(\nicep)=k$. Then the generating function $f_\nicep(t)$ (see (\ref{eqn:fp})) is of the form
$$f_\nicep(t)=\frac{h_\nicep(t)}{1-(-1)^{n/2}t^2}\cdot\prod_{i=1}^{k}\frac{1}{1-t^{2g_{i,n}}}$$
for some polynomial $h_\nicep(t)$.
\end{theorem}

Before proving this result first observe:
\begin{theorem}
Conjecture \ref{con:cyc-len} implies Conjecture \ref{thm:genfunmedium}.
\end{theorem}
\begin{proof}
The sequence $Z(P_m\times C_n)$ is a linear combination of sequences $Z(\nicep;m)$ for proper (in fact initial) patterns $\nicep$ of length $n$. Theorem \ref{thm:necklace-to-patterns} therefore implies that 
$$f_n(t)=\frac{h_n(t)}{1-(-1)^{n/2}t^2}\cdot\prod_{i=1}^{\lfloor n/4\rfloor}\frac{1}{1-t^{2g_{i,n}}}$$
for some polynomial $h_n(t)$. If Conjecture \ref{con:cyc-len} is true then we can take $g_{i,n}=n-3i$, thus obtaining the statement of Conjecture \ref{thm:genfunmedium}.
\end{proof}


\begin{proof}[Proof of Theorem \ref{thm:necklace-to-patterns}.]
For $k\geq 1$ and even $n$ let $\prop(k,n)$ denote the set of proper patterns $\nicep$ of length $n$ and such that $\mu(\nicep)=k$. Moreover, let $\prop_0(k,n)\subset\prop(k,n)$ consist of patterns which do not have any block in row $1$. These are exactly the reducible patterns.

There is a map
$$S:\prop(k,n)\to\prop_0(k,n)$$
defined as follows. If $\nicep\in\prop(k,n)$ then all blocks in the first row of $\nicep$ have length $3$. We apply the operation of type $N$ in the middle of every such block and define $S\nicep$ to be the resulting pattern. It has no blocks in the first row and $\mu(S\nicep)=\mu(\nicep)$ by Proposition~\ref{prop:propervn}.

Next we define a map
$$Q:\prop_0(k,n)\to\neck(k,n).$$

Note that a pattern $\nicep\in\prop_0(k,n)$ is determined by the positions of blocks in the second row and, for every endpoint of a block, the information whether the outermost $1$ in the first row (if any) is located above the $2$nd or $3$rd position of the block. This already determines the whole run above a block (because it is an alternating run of $0$'s and $1$'s).

Now we transcribe it into a necklace $Q\nicep$ as follows (see Fig.\ref{fig:neck2}). Label the unit intervals of a circle of length $n$ with the symbols from the second row of $\nicep$. For every block place two stones bounding that block and facing towards each other. The lengths of the vectors at those stones are determined by the rule:
\begin{itemize}
\item if the stone bounds a block of length $3$ then the length of its vector is $1$,
\item otherwise the length of a stone vector is the number of outermost $0$'s in the first row of $\nicep$ over the edge of the block bounded by the stone.
\end{itemize}
If two stones face away from each other then ``between them'' the second row of $\nicep$ contains a run $0101\cdots10$ of odd length. If two stones face towards each other then the length of the block between them is either $3$ or it is the odd length of $101\cdots01$ plus $p_1+p_2$ where $p_i$ are their stone vector lengths. It verifies that $Q\nicep$ is a $(k,n)$-necklace.

The map $Q$ is a bijection and we let
$$U:\neck(k,n)\to\prop_0(k,n)$$
be its inverse. More specifically, the second row of $UN$ is obtained by placing a block of $1$'s between every pair of stones that face each other and an alternating run $010\cdots10$ between stones facing away. In the first row of $UN$, over each block, we place either $0$'s (if the block has length $3$) or an alternating sequence $101\cdots01$ leaving out as many outermost positions as dictated by the stone vector lengths. We fill the remaining positions in the first row with $0$'s. The construction is feasible thanks to the parity conditions satisfied by $N$.

All the maps are defined in such a way that for every $(k,n)$-necklace $N$ we have
\begin{equation}
\label{claim-blah}
TN=QS((UN)^R)\quad \textrm{or equivalently}\quad UTN=S((UN)^R).
\end{equation}
To see this consider how the reduction operation $(\cdot)^R$ and the map $S$ change the neighbourhood of an endpoint of a block in the second row of $UN$. The argument is very similar to the proof of Proposition~\ref{prop:properr} and the details are left to the reader.

Now we complete the proof by induction on $\mu(\nicep)$. The case $\mu(\nicep)=0$ was dealt with in the proof in Section \ref{sect:weakproof}. Now suppose $k=\mu(\nicep)\geq 1$. If $\nicep\in\prop(k,n)$ is any pattern then by Propositions \ref{lem:zvn} and \ref{prop:propervn} we have an equation
$$Z(\nicep;m)=\pm Z(S\nicep;m)+\sum_\mathcal{R}\pm Z(\mathcal{R};m)$$
for some patterns $\mathcal{R}$ which satisfy $\mu(\mathcal{R})=k-1$. That means it suffices to prove the result for patterns in $\prop_0(k,n)$. Every such pattern is of the form $UN$ for some $(k,n)$-necklace $N$. Equation \eqref{claim-blah} and Propositions \ref{lem:zvn}, \ref{lem:zr} and \ref{prop:propervn} lead to an equation
\begin{eqnarray*}
Z(UN;m)&=&\pm Z((UN)^R;m-1)\\
&=&\pm Z(S((UN)^R);m-1)+\sum_\mathcal{R}\pm Z(\mathcal{R};m-1)\\
&=&\pm Z(UTN;m-1)+\sum_\mathcal{R}\pm Z(\mathcal{R};m-1)\end{eqnarray*}
with $\mathcal{R}$ as before. Now a $g_{k,n}$-fold iterated application of this argument for $N, TN, \ldots, T^{g_{k,n}}N=N$ produces an equation
$$Z(UN;m)=\pm Z(UN;m-g_{k,n})+\sum_\mathcal{R}\pm Z(\mathcal{R};m-b_\mathcal{R})$$
and its double application allows to avoid the problem of the unknown sign, that is we obtain
$$Z(UN;m)=Z(UN;m-2g_{k,n})+\sum_\mathcal{R'}\pm Z(\mathcal{R'};m-b_\mathcal{R'}).$$
It follows that the generating function of $Z(UN;m)$ can be expressed using combinations of the same rational functions which appeared in the generating functions for patterns of $\mu$-invariant $k-1$ together with $1/(1-t^{2g_{k,n}})$. That completes the proof.
\end{proof}

\newpage
\section{Appendix}
\subsection*{Table 1.} Some values of $Z(P_m\times C_n)$:	
\begin{center}
\begin{tabular}{l|ccccccccccccc}
$m$ $\backslash$ $n$ & 2 & 3 & 4 & 5 & 6 & 7 & 8 & 9 & 10 & 11 & 12 & 13 & 14 \\
\hline
0 & 1 & 1 & 1 & 1 & 1 & 1 & 1 & 1 & 1 & 1 & 1 & 1 & 1  \\
1 & -1 & -2 & -1 & 1 & 2 & 1 & -1 & -2 & -1 & 1 & 2 & 1 & -1  \\
2 & -1 & 1 & 3 & 1 & -1 & 1 & 3 & 1 & -1 & 1 & 3 & 1 & -1\\
3 & 1 & 1 & -3 & 1 & 1 & 1 & 5 & 1 & 1 & 1 & -3 & 1 & 1\\
4 & 1 & -2 & 5 & 1 & 4 & 1 & 5 & -2 & 1 & 1 & 8 & 1 & 1 \\
5 & -1 & 1 & -5 & 1 & -1 & 1 & 3 & 1 & 9 & 1 & -5 & 1 & -1\\
6 & -1 & 1 & 7 & 1 & 1 & 1 & 7 & 1 & -1 & 1 & 7 & 1 & 13\\
7 & 1 & -2 & -7 & 1 & 4 & 1 & 1 & -2 & 1 & 1 & 8 & 1 & 1\\
8 & 1 & 1 & 9 & 1 & 1 & 1 & 1 & 1 & 1 & 1 & 9 & 1 & 1\\
9 & -1 & 1 & -9 & 1 & -1 & 1 & -1 & 1 & -11 & 1 & -9 & 1 & 13\\
10 & -1 & -2 & 11 & 1 & 2 & 1 & 3 & -2 & -1 & 1 & 14 & 1 & -1\\
11 & 1 & 1 & -11 & 1 & 1 & 1 & -3 & 1 & 1 & 1 & -11 & 1 & 15\\
12 & 1 & 1 & 13 & 1 & 1 & 1 & 5 & 1 & 11 & 1 & 13 & 1 & 1
\end{tabular}
\end{center}

\subsection*{Table 2.} Some initial generating functions $f_{n}(t)$ for even $n$ are given below in reduced form. By $\Phi_k(t)$ we denote the $k$-th cyclotomic polynomial ($\Phi_1(t)=t-1$, $\Phi_2(t)=t+1$).

{\small
\begin{eqnarray*}
f_2(t)&=& \frac{-(t-1)}{\Phi_4(t)}\\
f_4(t)&=& \frac{-(t^2+1)}{\Phi_1(t)\Phi_2(t)^2}\\
f_6(t)&=& \frac{-(t^4+2t^3+2t+1)}{\Phi_1(t)\Phi_3(t)\Phi_4(t)}\\
f_8(t)&=& \frac{-(t^6-t^5+2t^4+6t^3+2t^2-t+1)}{\Phi_1(t)\Phi_2(t)^2\Phi_{10}(t)}\\
f_{10}(t)&=& \frac{q_{10}(t)}{\Phi_1(t)\Phi_4(t)\Phi_7(t)\Phi_8(t)}\\
f_{12}(t)&=& \frac{q_{12}(t)}{\Phi_1(t)\Phi_2(t)^2\Phi_3(t)\Phi_6(t)^2\Phi_{18}(t)}\\
f_{14}(t)&=& \frac{q_{14}(t)}{\Phi_1(t)\Phi_4(t)\Phi_5(t)\Phi_{11}(t)\Phi_{16}(t)}\\
f_{16}(t)&=& \frac{q_{16}(t)}{\Phi_1(t)\Phi_2(t)^2\Phi_{10}(t)\Phi_{14}(t)\Phi_{26}(t)}\\
f_{18}(t)&=& \frac{q_{18}(t)}{\Phi_1(t)\Phi_3(t)\Phi_4(t)\Phi_5(t)\Phi_8(t)\Phi_9(t)\Phi_{12}(t)\Phi_{15}(t)\Phi_{24}(t)}\\
f_{20}(t)&=& \frac{q_{20}(t)}{\Phi_1(t)\Phi_2(t)^2\Phi_4(t)\Phi_7(t)\Phi_{14}(t)\Phi_{22}(t)\Phi_{34}(t)}\\
f_{22}(t)&=& \frac{q_{22}(t)}{\Phi_1(t)\Phi_4(t)\Phi_7(t)\Phi_{13}(t)\Phi_{19}(t)\Phi_{20}(t)\Phi_{32}(t)}
\end{eqnarray*}
}

The longer numerators are as follows.
{\small
\begin{eqnarray*}
q_{10}(t)&=&-(t^{12}-t^{11}+t^8+9t^7+9t^5+t^4-t+1)\\
q_{12}(t)&=&-(t^{14}+2t^{13}+3t^{12}-3t^{11}+8t^{10}-5t^9+6t^8+6t^7+6t^6-5t^5+8t^4-3t^3+3t^2+2t+1)\\
q_{14}(t)&=&-(t^{24}-t^{19}+14t^{18}+14t^{17}+29t^{16}+42t^{15}+42t^{14}+55t^{13}+56t^{12}\\
&&		     +55t^{11}+42t^{10}+42t^9+29t^8+14t^7+14t^6-t^5+1)\\
q_{16}(t)&=&-(t^{28}-2t^{27}+5t^{26}+5t^{24}-2t^{23}+9t^{22}-7t^{21}+39t^{20}-37t^{19}\\
&&                   +44t^{18}-25t^{17}+30t^{16}-24t^{15}+26t^{14}-24t^{13}+30t^{12}\\
&&                   -25t^{11}+44t^{10}-37t^9+39t^8-7t^7+9t^6-2t^5+5t^4+5t^2-2t+1)\\
q_{18}(t)&=&-(t^{38}+2t^{37}-t^{36}+2t^{35}+6t^{34}-2t^{33}+2t^{32}+30t^{31}-2t^{30}\\                                                            
&&                   +t^{29}+70t^{28}-t^{27}+t^{26}+92t^{25}-t^{24}+t^{23}+130t^{22}-t^{21}\\                                                            
&&                   +168t^{19}-t^{17}+130t^{16}+t^{15}-t^{14}+92t^{13}+t^{12}-t^{11}+70t^{10}\\                                                        
&&                   +t^9-2t^8+30t^7+2t^6-2t^5+6t^4+2t^3-t^2+2t+1)\\
q_{20}(t)&=&-(t^{46}-2t^{45}+6t^{44}-10t^{43}+19t^{42}-18t^{41}+34t^{40}-40t^{39}\\
&&                   +64t^{38}-28t^{37}+60t^{36}-31t^{35}+120t^{34}-96t^{33}+189t^{32} \\
&&                   -147t^{31}+240t^{30}-195t^{29}+283t^{28}-230t^{27}+258t^{26} \\
&&                   -193t^{25}+218t^{24}-208t^{23}+218t^{22}-193t^{21}+258t^{20} \\
&&                   -230t^{19}+283t^{18}-195t^{17}+240t^{16}-147t^{15}+189t^{14} \\
&&                   -96t^{13}+120t^{12}-31t^{11}+60t^{10}-28t^9+64t^8-40t^7 \\
&&                   +34t^6-18t^5+19t^4-10t^3+6t^2-2t+1)\\
q_{22}(t)&=&-(t^{62}+t^{61}+t^{58}+t^{57}-t^{55}+22t^{54}+\cdots\cdots+22t^8-t^7+t^5+t^4+t+1)
\end{eqnarray*}
}

\subsection*{Table 3.} The decomposition of the directed graph $\neck(k,n)$ into a disjoint union of cycles. Here $l^p$ stands for $p$ copies of the cycle $C_l$.
\begin{center}
\begin{tabular}{l|ccccccc}
$n$ $\backslash$ $k$ & 1 & 2 & 3 & 4 & 5 & 6 & 7\\
\hline
4 & $1^1$ &  &  &  &  & &  \\
6 & $3^1$ &  &  &  &  & &  \\
8 & $5^1$ & $1^1$ &  &  &  &  & \\
10 & $7^1$ & $1^14^1$ &  &  &  & & \\
12 & $9^1$ & $2^13^26^1$ & $1^1$ &  & & & \\
14 & $11^1$ & $2^38^3$ & $1^15^1$ &  &  &  & \\
16 & $13^1$ & $2^55^310^3$ & $7^4$ & $1^1$ &  &  &\\
18 & $15^1$ & $1^12^912^6$ & $3^49^9$ & $1^26^1$ &  &  &\\
20 & $17^1$ & $2^{14}7^414^6$ & $11^{22}$ & $1^14^38^5$ & $1^1$ &  &\\
22 & $19^1$ & $2^{22}16^{10}$ & $13^{42}$ & $2^15^610^{21}$ & $1^27^1$ &  &\\
24 & $21^1$ & $2^{30}9^518^{10}$ & $3^15^{10}15^{70}$ & $2^53^{10}4^36^912^{63}$ & $1^13^19^9$ & $1^1$ & \\
26 & $23^1$ & $1^12^{42}20^{15}$ & $17^{120}$ & $2^{15}7^{28}14^{165}$ & $11^{49}$ & $1^38^1$ &\\
28 & $25^1$ & $2^{55}11^622^{15}$ & $19^{186}$ & $2^{43}8^{43}16^{378}$ & $13^{194}$ & $1^35^410^{11}$  & $1^1$\\
30 & $27^1$ & $2^{73}24^{21}$ & $7^{19}21^{270}$ & $2^{99}3^66^49^{90}18^{765}$ & $3^{26}5^{31}15^{613}$ & $1^13^54^{10}6^{13}12^{80}$  & $1^39^1$\\
32 & $29^1$ & $2^{91}13^726^{21}$ & $23^{396}$ & $2^{217}5^{43}10^{83}20^{1480}$ & $17^{1750}$ & $2^17^{52}14^{435}$  & $1^511^{17}$ \\
34 & $31^1$ & $1^12^{115}28^{28}$ & $5^125^{551}$  & $2^{429}11^{242}22^{2600}$   &  $19^{4334}$  & $2^74^{22}8^{124}16^{1791}$ & $1^113^{155}$\\
36 & $33^1$ & $2^{140}15^830^{28}$ & $9^{31}27^{738}$ & $1^12^{809}8^{12}12^{216}24^{4483}$ & $1^13^{18}7^{247}21^{9693}$ & $2^{31}3^{83}6^{76}9^{316}18^{6100}$ &  $1^13^{18}5^{31}15^{970}$
\end{tabular}
\end{center}


\end{document}